\documentclass[12pt]{article}
\usepackage{a4wide}
\usepackage{amsmath,amsfonts,amssymb}
\usepackage{amsthm}
\usepackage{graphicx}

\usepackage{xcolor}

\usepackage{hyperref}
\usepackage{breakurl}
\usepackage[T1]{fontenc}

\usepackage[numbers,sort&compress]{natbib}

\newcommand{\conv}[1]{\mathrm{conv}(#1)}
\newcommand{\leva}[1]{\mathrm{left}(#1)}
\newcommand{\prava}[1]{\mathrm{right}(#1)}
\newcommand{\out}[1]{\mathrm{out}(#1)}
\newcommand{\area}[1]{\mu(#1)}
\newcommand{\blue}[1]{\left\lVert#1\right\rVert_B}
\newcommand{\bluedashed}[1]{\left\lVert#1\right\rVert_{B'}}
\newcommand{\red}[1]{\left\lVert#1\right\rVert_R}
\newcommand{\reddashed}[1]{\left\lVert#1\right\rVert_{R'}}

\newtheorem{theorem}{Theorem}        
            
\newtheorem{lemma}[theorem]{Lemma}

\newtheorem{conjecture}[theorem]{Conjecture}

\theoremstyle{remark}

\begin{document}
	\title{Bicolored point sets admitting non-crossing alternating Hamiltonian paths 
	}

	\author{
		Jan Soukup\footnote{Department of Applied Mathematics, Charles University, Faculty of Mathematics and Physics, Malostransk\'e n\'am.~25, 118 00~ Praha 1, Czech Republic; \texttt{soukup@kam.mff.cuni.cz}. Supported by project 23-04949X of the Czech Science Foundation (GA\v{C}R) and by the grant SVV–2023–260699.  
		}
	} % end of author
	\maketitle              % typeset the header of the contribution
	\begin{abstract}
		Consider a bicolored point set $P$ in general position in the plane consisting of $n$ blue and $n$ red points. We show that if a subset of the red points forms the vertices of a convex polygon separating the blue points, lying inside the polygon, from the remaining red points, lying outside the polygon, then the points of $P$ can be connected by non-crossing straight-line segments so that the resulting graph is a properly colored closed Hamiltonian path.
		
		%\keywords{geometric graph  \and bichromatic point set \and non-crossing alternating path \and graph drawing \and Hamiltonian path}
	\end{abstract}

	\section{Introduction}

	%A graph $G = (V, E)$ is an ordered pair where $V$ is a set, whose elements are called \emph{vertices}, and $E$ is a set of unordered pairs of vertices, whose elements are called \emph{edges}. 
	In geometric graph theory it is a common problem to decide whether a given graph can be drawn in the plane on a given point set so that the edges are represented by non-crossing straight-line segments. For example, deciding whether a given general planar graph has a~non-crossing straight-line drawing on a given point set is NP-complete \cite{embedding_NP}. 
	
	There are many interesting unanswered questions when considering bicolored point sets instead (see the comprehensive survey by \citet{Survey21}). We restrict ourselves to drawings of bipartite graphs on bicolored point sets where edges are drawn as non-crossing straight-line segments between points of different colors. %Then the question of embeddability of $G$ on a bicolored point set $S$ is equivalent to finding a non-crossing copy of $G$ in a complete geometric bipartite graph on $S$ whose bipartition is formed by the color classes of $S$.
	This question remains interesting even for paths. Let $B$ and $R$ denote a set of blue points and a set of red points in the plane, respectively, such that $R \cup B$ is in general position, i.e., no three points are collinear. We call a non-intersecting path on $R\cup B$ whose edges are straight-line segments and every segment connects two points of $R\cup B$ of opposite colors, an $\emph{alternating path}$. If such an alternating path connects all points of $R\cup B$, we call it an \emph{alternating Hamiltonian path}. If such an alternating Hamiltonian path shares the first and last vertex (but otherwise is still non-intersecting), we call it a \emph{closed alternating Hamiltonian path}.
	
	If $||R| - |B|| \le 1$ and $R$ can be separated from $B$ by a line, then \citet{straitline} showed that there always exists an alternating Hamiltonian path on $R \cup B$. This fact together with the well-known Ham sandwich theorem implies that if $|R| = |B|$, then there always exists an alternating path on $R \cup B$ connecting at least half of the points. This trivial lower bound on the length of an alternating path that always exists is the best known according to our knowledge. This bound was improved by a small linear factor by \citet{Long_paths} for point sets in \emph{convex} position, i.e., when the points form the vertices of a convex polygon. On the other hand, if we do not assume that $R$ and $B$ are separated by a line, then there are examples where $|R| = |B| \ge 8$ and no alternating Hamiltonian path on $R \cup B$ exists, even if $R\cup B$ is in convex position. Moreover, for $R\cup B$ in convex position with $|R| = |B| = n$, \citet{Necklace} showed that there are configurations where the longest alternating path on $R\cup B$ has size at most $(4-2\sqrt{2})n  + o(n)$.
	
	As we have seen above, an alternating Hamiltonian path does not exist on every point set but it exists if $||R| - |B|| \le 1$ and $R$ can be separated from $B$ by a line. Several other sufficient conditions on the point set where an alternating Hamiltonian path exists are known. \citet{doublechain} looked more closely at configurations where $R$ and $B$ form a double chain. A \emph{convex} or a~\emph{concave chain} is a finite set of points in the plane lying on the graph of a strictly convex or a strictly concave function, respectively. A \emph{double-chain} consists of a convex chain and a concave chain such that each point of the concave chain lies strictly below every line determined by the convex chain and, similarly, each point of the convex chain lies strictly above every line determined by the concave chain. \citet{doublechain} showed that if $||R| - |B|| \le 1$ and each of the chains of the double-chain contains at least one-fifth of all points, then there exists an alternating Hamiltonian path on $R \cup B$. Moreover, they showed that such a path does not always exist if one chain contains approximately 28 times more points than the other.
	
	Another sufficient condition for the existence of an alternating Hamiltonian path was found by \citet{straitline}. They showed that if $\left||R| - |B|\right| \le 1$, the points of $R$ are vertices of a convex polygon, and all points of $B$ are inside this polygon, then there exists an alternating Hamiltonian path on $R\cup B$. There is no other sufficient condition that we know of.
	
	In this paper, we generalize this last result, and by doing so, we extend the known family of configurations of points for which there exists an alternating Hamiltonian path on $R \cup B$. Specifically, we prove the following theorem.
	
	\begin{theorem} \label{mainpolygon}
		Let $R$ be a set of red points and $B$ be a set of blue points such that $R \cup B$ is in general position. Let $P$ be a convex polygon whose vertices are formed by a subset of~$R$. Assume that the remaining points of~$R$ lie outside of $P$, points of $B$ lie in the interior of~$P$, and $\left||R|- |B|\right| \le 1$. Then there exists an alternating Hamiltonian path on $R \cup B$.
	\end{theorem}
	
	When $|R|= |B|$ we even find a closed alternating Hamiltonian path, i.e., a cycle. Hamiltonian cycles were investigated in a similar setting by \citet{cycles}. They proved that if we relax the alternating condition to allow for at most $n-1$ crossings, then we can always find an alternating Hamiltonian cycle, and that this is the minimum number of crossings needed.
	
	\section{Preliminaries and an outline of the proof} \label{chap01}
	By a \emph{polygonal region} we understand a closed, possibly unbounded, region in the plane whose boundary (possibly empty) consists of finitely many non-crossing straight-line segments or half-lines connected into a polygonal chain. A bounded polygonal region is a polygon. A polygon can be defined by an ordered set of its vertices; in that case, we assume that the vertices lie on the boundary of the polygon in the clockwise direction, and we use index arithmetic modulo the number of vertices. A \emph{diagonal} of a convex polygon (or a polygonal region) is any segment connecting two points on the boundary of the polygon. The \emph{convex hull} of a set of points $X$, denoted by $\conv{X}$, is the smallest convex set that contains~$X$. Recall that $B$ and $R$ always denote the set of blue points and the set of red points, respectively. Moreover, $B$ and $R$ are always disjoint, and $R \cup B$ is always in general position.
	
	Our primary result, Theorem~\ref{mainpolygon}, is a generalization of the following theorem proved by \citet{straitline}.
	
	\begin{theorem} [\cite{straitline}] \label{originalpoly}
		Let $R$ be a set of red points and $B$ be a set of blue points such that $R \cup B$ is in general position. Let $R$ form the vertices of the polygon $\conv{R \cup B}$, the points of $B$ lie in the interior of $\conv{R\cup B}$, and $\left||R|- |B|\right| \le 1$. Then there exists an alternating Hamiltonian path on $R \cup B$.
	\end{theorem}
	
	Our improvement lies in the fact that the polygon $P$ can be formed by a subset of $R$ (instead of the whole $R$), whereas the remaining points of $R$ remain outside of $P$. The approach in the proof of Theorem~\ref{originalpoly} in a case when $|R| = |B|$ is to partition the polygon formed by $R$ into convex parts, each containing exactly one edge of the polygon and one blue point from inside the polygon, and then connect by straight-line segments each of the blue points to the vertices of the edge that is inside the same part. In this way, alternating paths of length two are formed inside each part of the partition. Moreover, they share their end vertices, and so together they form a closed alternating Hamiltonian path (this path is non-crossing since each of the small paths lies in its own part of the partition). 
	
	We proceed similarly with only one significant distinction. Namely, we partition the whole plane into convex parts such that every edge of the polygon is a diagonal of one part of the partition, and each part of the partition contains one more blue point than it contains red points (not counting the vertices of the polygon). Inside each of these parts, we find an alternating Hamiltonian path. And these paths together form a closed alternating Hamiltonian path as before.
	
	Before we begin, we introduce some geometric notation needed in our arguments. For a directed line $l$, the closed half-plane to the left of $l$ is denoted by $\leva{l}$, and the closed half-plane to the right of $l$ is denoted by $\prava{l}$. For an edge $e$ of a convex polygon (or polygonal region), the closed half-plane to the side of $e$ that is disjoint with the polygon's interior is denoted by $\out{e}$. For two points $a, b$ in the plane, we denote by $ab$ the segment connecting them. 
	
	For a region $T$ of the plane, $\red{T}$ and $\blue{T}$ denotes the number of red points inside~$T$ and the number of blue points inside~$T$, respectively. For the points on the boundaries of regions, we specify if they belong to the region or not (we want each point to belong to exactly one region of the partition).

	\section{Partitioning theorem }
	For the partitioning of the plane, we prove the following theorem.

	\begin{theorem} \label{chp01:partition}
		Let $P = (p_1, \dots ,p_s)$ be a convex polygon, $B$ be a set of blue points in the interior of $P$, and $R$ be a set of red points outside of $P$ such that ${s=|B| - |R|}$ and $R \cup B \cup \{p_1, \dots p_s\}$ is in general position. Then there exists a partition of the plane into convex polygonal regions $Q_1, \dots, Q_s$ such that each $p_ip_{i+1}$ is a diagonal of $Q_i$ and for every $i$, we have $\blue{Q_i} - \red{Q_i} = 1$ (index arithmetic is modulo $s$). Moreover, every point of $R \cup B$ is counted in exactly one $Q_i$. That is, if a point of $R \cup B$ lies on the common boundary of more $Q_i$'s it is assigned to only one of them.   
	\end{theorem}
	
	We believe that the following stronger conjecture, where the differences $\blue{Q_i} - \red{Q_i}$ are predefined integers satisfying some conditions, holds. We formulate it not only for partitioning of the plane but for partitioning of any convex polygonal region since it will be helpful later.
	\begin{conjecture} \label{cp01:conjecture}
		Let $Q$ be a convex polygonal region, $P = (p_1, \dots ,p_s)$ be a convex polygon inside $Q$, $B$ be a set of blue points in the interior of $P$, and $R$ be a set of red points outside $P$ but inside $Q$ such that $R \cup B \cup \{p_1, \dots p_s\}$ is in general position. Additionally let $n_1, \dots, n_s$ be integers satisfying the following conditions. 
		\begin{enumerate}
			\item $ |B| - |R| = n_1 + \dots + n_s.$
			\item For every nonempty cyclic interval of indices $I$ from $\{1, \dots, s\}$, 
			\begin{equation} \label{cp01:conjectureeq}
				\sum_{i \in I} n_i \ge -\red{Q \cap\bigcup_{i \in I} \out{p_ip_{i+1}}}.
			\end{equation}
		\end{enumerate} 
		Then there exists a partition of $Q$ into convex polygonal regions $Q_1, \dots, Q_s$ such that for every $i$, the segment $p_ip_{i+1}$ is a diagonal of $Q_i$ and $\blue{Q_i} - \red{Q_i} = n_i$. Moreover, every point of $R \cup B$ is counted in exactly one $Q_i$.
	\end{conjecture}
	
	\begin{figure}[htb]
		\centering
		\includegraphics[scale=1]{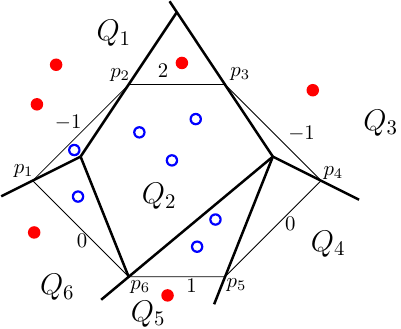}
		\caption[A partition of $Q$ into convex polygonal regions $Q_1,\dots Q_6$]{A partition of $Q$ into convex polygonal regions $Q_1,\dots Q_6$ as in Conjecture~\ref{cp01:conjecture}. Blue points are drawn as circles and red points as disks. Numbers $n_i$'s are written next to their corresponding edges of the polygon.}\label{fig1:conjecture}
	\end{figure}
	
	For an example partition, see Figure~\ref{fig1:conjecture}.

	Note that Conditions~ \eqref{cp01:conjectureeq} are necessary: When $I=\{i\}$, the part $Q_i$ is split by $p_ip_{i+1}$ into two regions, one inside $P$ and one outside of $P$. The part outside of $P$ is inside $\out{p_ip_{i+1}}$, and so $\red{Q_i} \le \red{Q\cap \out{p_ip_{i+1}}}$. Together with a trivial condition $\blue{Q_i}\ge 0$ we get $\blue{Q_i} - \red{Q_i}\ge -\red{Q\cap \out{p_ip_{i+1}}}$, which is exactly one of the condition in Conditions~\eqref{cp01:conjectureeq}. It can be analogously observed for larger cardinalities of  $I$'s. 
	
	Furthermore, note that in the case when $Q$ is the plane and all $n_i$'s are equal to 1, Conditions~\eqref{cp01:conjectureeq} always hold, and the conjecture is equivalent with Theorem~\ref{chp01:partition}.
	
	The case when there are no red points outside of $P$ and every $n_i$ is a positive integer was already proved by \citet{partitioning} and later by \citet{partitioningnew}. The case with points outside of $P$ seems to be more difficult (for example, even some negative $n_i$'s can satisfy Conditions \eqref{cp01:conjectureeq} in that case).
	
	Similar problems of finding partitions of colored point sets into subsets with disjoint convex hulls such that the sets of points of all color classes are partitioned as evenly as possible is well studied, see \cite{equi_partition, parition_to_more}. However, we were not able to apply the results directly because we have the additional restriction that $p_ip_{i+1}$'s have to be diagonals of the convex hulls in the partition.  
	
	We managed to prove Conjecture~\ref{cp01:conjecture} only in a case when $s=3$, but that proved crucial in proving Theorem~\ref{chp01:partition}.
	\begin{lemma} \label{chp01:cases3}
		For $s=3$, Conjecture~\ref{cp01:conjecture} holds.
	\end{lemma}
	
	In the proof, we employ a standard technique (see \citet{points_to_discs}) and substitute points with disks of the same area and work with the area of the disks instead of the number of points. This is helpful because the boundaries of polygonal regions have an area of size zero, and so the area of all disks will be precisely distributed between the interiors of the polygonal regions of the partition. At the end of the proof, we return from disks back to points and we have to solve the problem where to assign points whose disks intersect the boundaries of the partition.
	
	We will also use the following Knaster–Kuratowski–Mazurkiewicz lemma.
	
	\begin{lemma}[Knaster–Kuratowski–Mazurkiewicz lemma]\label{chp01:KKM}~\\
		Let $S = \conv{\{e_1, e_2 , e_{3}\}} \subset \mathbb{R}^{2}$ and $\{F_1, F_2,  F_{3}\}$ be a family of closed subsets of $S$ such that for $A \subseteq \{1, 2, 3\}$, we have 
		
		\[\conv{\{e_i: i \in A\}} \subseteq \bigcup_{i \in A} F_i.\]
		
		Then $\bigcap_{i=1}^{3} F_i$ is compact and non-empty.
	\end{lemma}
	
	For a simple proof of this lemma, see \cite[Theorem 5.1]{KMMlemma}.	This lemma also holds in an analogous form in higher dimensions, but we need just the planar version.

	\begin{proof} [Proof of Lemma~\ref{chp01:cases3}]
		For a point $x \in P \setminus \{p_1, p_2\}$, define $Q^x_1$ as the convex polygonal region enclosed by the boundary of $Q$ and by the half-lines $xp_1, xp_2$ so that $Q^x_1$ contains edge $p_1p_2$ (or, in a degenerate case when $x$ lies on $p_1, p_2$, define it as $\out{p_1 p_2} \cap Q$). Define $Q^x_2$ and $Q^x_3$ analogously.
		
		Let $\{b_1, \dots, b_{|B|}\} = B$ and $\{r_1, \dots, r_{|R|}\} = R$. We substitute every $b_i$ by a disk $b'_i$ with $b_i$ in its center and substitute every $r_i$ by a disk $r'_i$ with $r_i$ in its center. Furthermore, since the set $R \cup B \cup \{p_1, p_2, p_3\}$ is in general position we can do the substitution so that every disk has a same positive diameter $\varepsilon$, and no line intersects more than two disks or vertices of $P$ simultaneously. Let $\area{T}$ denote the multiple of the area (standard Lebesgue measure) of a region $T$ in the plane such that for each of our disks $d$, we have $\area{d}=1$. 
		
		For all $i$, $1\le i \le 3$, let 
		\[F_i = \left\{x \in P \setminus \{p_i, p_{i+1}\}: \sum_{j=1}^{|B|}\area{Q^x_i \cap b'_j} - \sum_{j=1}^{|R|}\area{Q^x_i \cap r'_j} \ge n_i\right\},\]
		see Figure~\ref{fig2:cases3fi} for an example.	
		
		\begin{figure}[tbh]
			\centering
			\includegraphics[scale=1]{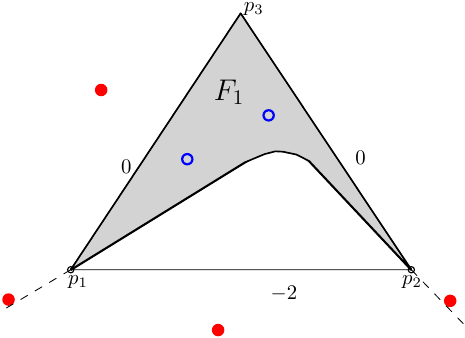}
			\caption{Approximate image of a set $F_1$ used by Lemma~\ref{chp01:KKM} to find a partition of the triangle in Lemma~\ref{chp01:cases3}.}
			\label{fig2:cases3fi}
		\end{figure}
		
		Take $x$ from $P \setminus \{p_1, p_2, p_3\}$. Since $|B| - |R| = n_1 +n_2 + n_3$ and for every such $x$, $(Q^x_1, Q^x_2, Q^x_3)$ is a partition of $Q$, we have 
		\begin{equation} \label{chp01:in1}
			\sum_{i=1}^{3}\sum_{j=1}^{|B|}\area{Q^x_i \cap b'_j} - \sum_{i=1}^{3}\sum_{j=1}^{|R|}\area{Q^x_i \cap r'_j} = |B| - |R| = n_1 + n_2 + n_3.
		\end{equation}
		Therefore, $F_1 \cup F_2 \cup F_3$ covers the interior of $P$.\\
		
		Take $x \in p_1p_2 \setminus \{p_1, p_2\}$. Then $\out{p_1p_2}\cap Q = Q^x_1$ and $\sum_{j=1}^{|B|}\area{Q^x_1 \cap b'_j} = 0$ since all the blue disks are in the interior of $P$. Moreover, by Conditions \eqref{cp01:conjectureeq} and the fact that the line $p_1p_2$ does not cross any disks, we have $\sum_{j=1}^{|R|}\area{Q^x_1 \cap r'_j} \ge -n_1$. Summing this with Equation \eqref{chp01:in1} we have
		
		\begin{align*}
			%&\sum_{i=1}^{3}\sum_{j=1}^{|B|}||Q^x_i \cap b'_j|| - \sum_{i=1}^{3}\sum_{j=1}^{|R|}||Q^x_i \cap r'_j|| + n_1 \ge n_1 + n_2 + n_3 - \sum_{j=1}^{|R|}||Q^x_1 \cap r'_j||\\
			&\sum_{j=1}^{|B|}\area{Q^x_2 \cap b'_j} + \sum_{j=1}^{|B|}\area{Q^x_3 \cap b'_j} - \sum_{j=1}^{|R|}\area{Q^x_2 \cap r'_j} - \sum_{j=1}^{|R|}\area{Q^x_3 \cap r'_j} \ge n_2 + n_3	.		
		\end{align*}
		Hence, $x \in F_2$ or $x \in F_3$. Therefore, $F_2 \cup F_3$ covers $p_1p_2$. Analogously, we see that $F_1 \cup F_2$ covers $p_3p_1$ and $F_3 \cup F_1$ covers $p_2p_3$.\\
		
		Take $x=p_1$. Then $Q$ can be partition into two polygonal regions, $Q^x_2$ and $Q\cap (\out{p_1, p_2} \cup \out{p_3, p_1})$. Moreover, $Q^x_2$ contains all blue disks. Thus,
		\begin{align*}
			\sum_{j=1}^{|B|}\area{Q^x_2 \cap b'_j} - \sum_{j=1}^{|R|}\area{Q^x_2 \cap r'_j} - & \red{Q\cap (\out{p_1, p_2} \cup \out{p_3, p_1})} = \\&= |B| - |R|
			= n_1 + n_2 + n_3.
		\end{align*}
		
		By Conditions \eqref{cp01:conjectureeq} we have \[\red{Q\cap (\out{p_1, p_2} \cup \out{p_3, p_1})} \ge -n_1-n_3.\]
		Hence, we see that
		\[\sum_{j=1}^{|B|}\area{Q^x_2 \cap b'_j} - \sum_{j=1}^{|R|}\area{Q^x_2 \cap r'_j}  \ge  n_2.\] 
		Therefore, $F_2$ covers $p_1$. Analogously, $F_3$ covers $p_2$ and $F_1$ covers $p_3$.
		
		Furthermore, all sets $F_i$ are closed, except possibly each $F_i$ in points $p_{i}, p_{i+1}$ (index arithmetic is modulo 3). But since $F_{i}$ covers $p_{i-1}$ it has to cover even a small disk around~$p_{i-1}$ because points are in general position. Thus, we can remove this open disk from~$F_{i-1}$ and~$F_{i+1}$. If we do this for all $F_i$, then all of them will be closed and we can apply Lemma~\ref{chp01:KKM}.
		
		Therefore, there exists a point $y \in P$ such that $y \in F_1 \cap F_2 \cap F_3$. The point $y$ is not a vertex of $P$, because each vertex of $P$ is in only one of $F_i$'s. We claim that the polygonal regions $Q^y_1, Q^y_2$, and $Q^y_3$ form the desired partition of $Q$. Clearly, every $p_ip_{i+1}$ is a diagonal of $Q_i^y$. Since the area of the intersection of any two $Q^y_i, Q^y_j$ is zero, then, by Equality \eqref{chp01:in1} and the definitions of $F_i$, for every $i$ we have, 
		\begin{equation} \label{chp01:in2}
			\sum_{j=1}^{|B|}\area{Q^y_i \cap b'_j} - \sum_{j=1}^{|R|}\area{Q^y_i \cap r'_j} = n_i.
		\end{equation} It remains to show that corresponding equalities holds also for points. That is, we need to show that for every $i$ we have
		\begin{equation} \label{chp01:in3}
			\blue{Q^y_i} - \red{Q^y_i} = n_i.
		\end{equation}
		
		Since the points are in general position, every line $yp_i$ crosses at most one disk. Therefore, there are only few possibilities of how the intersection of these lines with disks can look like.
		
		\begin{figure}[htb]
			\centering
			\includegraphics[scale=1]{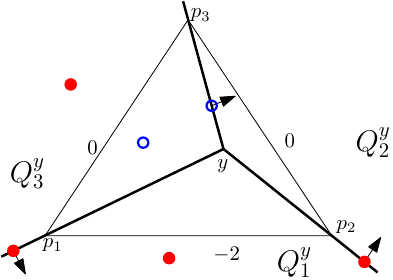}
			\caption{Assignment of points when disks are intersected by boundaries of polygonal regions.}
			\label{fig2:cases3disk}
		\end{figure}
		
		\begin{itemize}
			\item No line $yp_i$ crosses any disks:
			
			In this case every disk is entirely inside some $Q^y_i$. Therefore,  
			\[\sum_{j=1}^{|B|}\area{Q^y_i \cap b'_j} - \sum_{j=1}^{|R|}\area{Q^y_i \cap r'_j} = \blue{Q_i^y} - \red{Q_i^y}\] and the theorem follows.
			
			\item Exactly one line, without loss of generality $yp_1$, crosses some disk:
			
			In this case $\sum_{j=1}^{|B|}\area{Q^y_1 \cap b'_j} - \sum_{j=1}^{|R|}\area{Q^y_1 \cap r'_j}$ is not an integer, a contradiction with \eqref{chp01:in2}.
			
			\item Two lines, without loss of generality $yp_1, yp_{2}$, cross some disks:
			
			In this case $\sum_{j=1}^{|B|}\area{Q^y_2 \cap b'_j} - \sum_{j=1}^{|R|}\area{Q^y_2 \cap r'_j}$ is not an integer, a contradiction with \eqref{chp01:in2}.
			
			\item All three lines $yp_1, yp_{2}, yp_3$, cross some disks:
			
			If two of them, say $yp_1, yp_2$, cross the same disk, then $\sum_{j=1}^{|B|}\area{Q^y_1 \cap b'_j} - \sum_{j=1}^{|R|}\area{Q^y_1 \cap r'_j}$ is not an integer, a contradiction with \eqref{chp01:in2}.
			
			Thus, each of them crosses a different disk, say $yp_i$ crosses $d'_i$ with the colored point $d_i$ in its center. Hence, $Q_1^y$ intersects $d'_1$ and $d'_2$, $Q^y_2$ intersects $d'_2$ and $d'_3$, and $Q_3^y$ intersects $d'_3$ and $d'_1$. Firstly, assume that some $d_i$ does not lie on its corresponding line $yp_i$. Then by Equalities \eqref{chp01:in2} even the centers of the other two $d'_i$ cannot lie on their corresponding lines. Thus, all points of $R \cup B$ are in the interiors of some $Q^y_i$'s. Therefore, if we round to the nearest integer the contribution of every disk $d'_k$ to the value $\sum_{j=1}^{|B|}\area{Q^y_1 \cap b'_j} - \sum_{j=1}^{|R|}\area{Q^y_1 \cap r'_j}$, we obtain the value $\blue{Q^y_1} - \red{Q^y_1}$. Furthermore, in $Q_1^y$ only disks $d_1'$ and $d_2'$ could have changed their contribution. Thus, we could not have changed the value by $1$ or more by the rounding. Therefore,
			\[\sum_{j=1}^{|B|}\area{Q^y_1 \cap b'_j} - \sum_{j=1}^{|R|}\area{Q^y_1 \cap r'_j} = \blue{Q_1^y} - \red{Q_1^y}\]
			since both sides are integers.
			Analogously, the same holds for $Q_2^y$ and $Q_3^y$, and so Equations \eqref{chp01:in3} are satisfied.

			It remains to solve the case when each $d_i$ lies on its corresponding line $yp_i$. In this case, each $d_i$ lies on the boundaries of two $Q^y_i$'s, and we have to assign each of them to exactly one $Q^y_i$ such that Equations \eqref{chp01:in3} would hold. If all $d_i$ have the same color, we assign one to every $Q^y_i$. If two of them, say $d_1$ and $d_2$ are red, and the remaining one, $d_3$ in our case, is blue we assign $d_2$ and $d_3$ to $Q_2^y$ and we assign $d_1$ to $Q_1^y$ (see Figure~\ref{fig2:cases3disk}). The case for two blue and one red is analogous. With these assignments we have
			\[\sum_{j=1}^{|B|}\area{Q^y_i \cap b'_j} - \sum_{j=1}^{|R|}\area{Q^y_i \cap r'_j} = \blue{Q_i^y} - \red{Q_i^y}\]
			for every $i$. Hence, Equations \eqref{chp01:in3} are satisfied.
		\end{itemize}

	\end{proof}

	\begin{sloppypar}

        To prove Theorem~\ref{chp01:partition}, we use induction on the number of vertices of the polygon~$P$. The main idea is to find a suitable triangle formed by vertices of $P$, and apply Lemma~\ref{chp01:cases3} to this triangle. We set the numbers $n_1, n_2, n_3$ so that we obtain three polygonal regions $Q_1$, $Q_2$ and $Q_3$ each containing $\blue{Q_i} - \red{Q_i}$ edges of $P$. Then we partition $Q_1$, $Q_2$ and $Q_3$ by induction. Note that except for the very first step, we are partitioning bounded polygonal regions instead of the plane, and the polygon $P$ is already partially split but that is only easier. %See Figure~\ref{fig3:both}, left.

		%To prove Theorem~\ref{chp01:partition}, we want to use induction on the number of vertices of the polygon~$P$. The main idea is to find a suitable triangle formed by vertices of $P$ and apply Lemma~\ref{chp01:cases3} to this triangle. This way we obtain three polygonal regions that have to be partitioned further, which we do by induction.
		
		Unfortunately, as we will see later, at the beginning of this process, Conditions \eqref{cp01:conjectureeq} in Lemma~\ref{chp01:cases3} pose a problem and finding the triangle is not always possible. But for now, we present the main induction argument in full details.
		
	\end{sloppypar}
	
	\begin{lemma} \label{chap01:partiitoninduction}
		\begin{sloppypar}
			Let $Q$ be a convex polygonal region with a boundary and $a$ be a point on this boundary. Additionally, let $s\ge1$ be an integer and $p_1, p_2, \dots, p_{s+1}$ be points inside $Q$ such that $p_1$ and $p_{s+1}$ lie on the boundary of $Q$, $a$ lies in $\prava{p_1p_{s+1}}$, and $(a, p_1, \dots, p_{s+1})$ is a convex polygon inside $Q$ (we also allow cases when either $p_1 = a$, or $p_{s+1} = a$). Let $P$ be the convex polygon enclosed by $p_1p_2,  \dots ,p_sp_{s+1}$ and by the part of the boundary of $Q$ from $p_{s+1}$ through $a$ to $p_1$ in clockwise direction. Moreover, let $B$ be a set of blue points inside $P$, and $R$ be a set of red points outside $P$ but inside $Q$ such that $s = |B| - |R|$. Assume that $R \cup B \cup \{a, p_1, p_2, \dots p_{s+1}\}$ is in general position.
		\end{sloppypar}
		
		Then there exists a partition of $Q$ into convex polygonal regions $Q_1, \dots, Q_s$ such that $p_ip_{i+1}$ is a diagonal of $Q_i$ and for every $i$, we have $\blue{Q_i} - \red{Q_i} = 1$. Moreover, every point of $R \cup B$ is counted in exactly one $Q_i$.
	\end{lemma}
	
	Note that instead of partitioning the plane as in Theorem~\ref{chp01:partition} we are partitioning a polygonal region $Q$. And instead of a polygon $P$ inside $Q$, we have a chain of vertices attached to the boundary of $Q$ on both ends. It will immediately follows, that after we manage to split the plane into three polygonal regions, we can apply this lemma to prove Theorem~\ref{chp01:partition}.
	
	It is important that the polygonal chain $(p_1, \dots, p_{s+1})$ is attached to the boundary of~$Q$ because Conditions \eqref{cp01:conjectureeq} does not pose a problem in this case as will be clear from the proof.
	
	\begin{proof}[Proof of Lemma~\ref{chap01:partiitoninduction}] 		
		We use induction on $s$.
		
		\begin{sloppypar}
			First assume that $s=1$. In this case, set $Q_1 = Q$ and assign all points of $B\cup R$ on the boundary of $Q$ to $Q_1$. Since $s = |B| - |R|$, we have $\blue{Q_1} - \red{Q_1} = 1$.
		\end{sloppypar}
		
		Now assume that $s\ge2$.
		For every $p_i$, $2\le i \le s$, we say that $p_ia$ is left-partitionable if 
		\[\blue{Q \cap \prava{p_ia}} - \red{Q \cap \prava{p_ia}} \le i-1\]
		and we say that $p_ia$ is right-partitionable if 
		\[\blue{Q \cap \prava{p_ia}} - \red{Q \cap \prava{p_ia}} \ge i-1.\]
		Since $s = |B| - |R|$ and the points $p_i$ together with $a$ are in general position, $p_ia$ is left-partitionable if 
		\[\blue{Q \cap \leva{p_ia}} - \red{Q \cap \leva{p_ia}} \ge s -i +1\]
		and $p_ia$ is right-partitionable if 
		\[\blue{Q \cap \leva{p_ia}} - \red{Q \cap \leva{p_ia}} \le s -i +1.\]
		Moreover, every such $p_ia$ is left or right-partitionable.
		
		Assume that $p_2a$ is right-partitionable. Note that this does not happen if $a=p_1$ because there is no blue point to the right of $p_2p_1$. Let $l$ be a line containing $p_2a$. By rotating $l$ in the clockwise direction around $p_2$ we can find a point $x$ on the boundary of $Q$ between $a$ and $p_1$ different from $p_1$ such that \[\blue{Q \cap \prava{p_2x}} - \red{Q \cap \prava{p_2x}} = 1\] since \[\blue{Q \cap \prava{p_2p_1}} - \red{Q \cap \prava{p_2p_1}} \le 0\] (again note that there are no blue or red points on the line $p_2p_1$ because ${R \cup B \cup \{p_2, p_1\}}$ is in general position).
		Thus, we can set $Q_1 = Q\cap \prava{p_2x}$ and apply induction on the polygonal region $Q \cap \leva{p_2x}$ with $a$ on its boundary, points $p_2,\dots p_{s+1}$, and the remaining red and blue points. Then the partition of $Q \cap \leva{p_2x}$ together with $Q_1$ forms the desired partition of $Q$. See Figure~\ref{fig2.5:inductionstep_basic}.
		
		\begin{figure}[htb]
			\centering
			\includegraphics[scale=1]{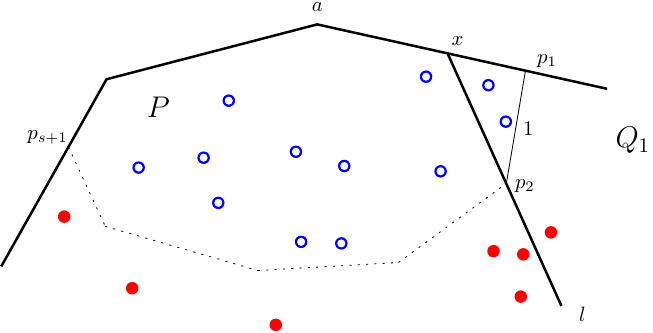}
			\caption[Induction step in Lemma~\ref{chap01:partiitoninduction} when $p_2a$ is right-partitionable.]{Induction step in Lemma~\ref{chap01:partiitoninduction} when $p_2a$ is right-partitionable. Polygon $Q_1$ is the first part of the partition of $Q$. The remaining parts of the partition are obtained by induction hypothesis applied on the remaining part of $Q$.}
			\label{fig2.5:inductionstep_basic}
		\end{figure}
		
		Therefore, we can assume that $p_2a$ is left-partitionable. By similar analysis for $p_sa$ we solve the case when $p_sa$ is left-partitionable. Hence, we can assume that $p_sa$ is right-partitionable.
		
		\begin{sloppypar}
			Thus, we can find index $j$, $2\le j \le s-1$ such that $p_ja$ is left-partitionable and $p_{j+1}a$ is right-partitionable. Let $T$ be the triangle $ap_jp_{j+1}$. Let ${B' = B \cap T}$, ${n(ap_j) = j-1 - \blue{\prava{p_ja}}}$,  ${n(p_{j+1}a) = s -j - \blue{\leva{p_{j+1}a}}}$, and  $n(p_jp_{j+1}) = 1$. We want to use Lemma~\ref{chp01:cases3} on $Q$, triangle $T$ inside $Q$, set of blue points $B'$, set of red points $R$, and the numbers $n(ap_j), n(p_{j+1}a), n(p_jp_{j+1})$.
		\end{sloppypar}
		
		Clearly, $B'\cup R \cup \{a, p_j, p_{j+1}\}$ are in general position, $B'$ is in the interior of $T$, $R$ is outside $T$, and $|B'| - |R| = n(ap_j) + n(p_{j+1}a) + n(p_jp_{j+1})$. See Figure~\ref{fig3:inductionstep} for an illustration. It remains to check that Conditions $\eqref{cp01:conjectureeq}$ hold.
		
		\begin{figure}[!htb]
			\centering
			\includegraphics[scale=1]{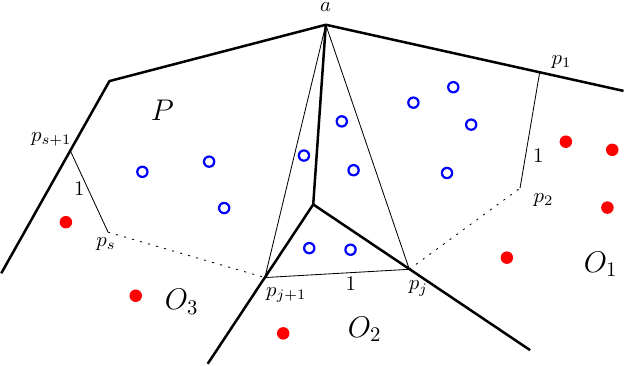}
			\caption[Induction step in Lemma~\ref{chap01:partiitoninduction}.]{Induction step in Lemma~\ref{chap01:partiitoninduction}. Polygonal region is split into three parts $O_1, O_2, O_3$. Induction hypothesis can then be applied on $O_1$ and $O_3$ to obtain a complete partition of the polygonal region.}
			\label{fig3:inductionstep}
		\end{figure}
		
		Let us first check it for the edge $p_ja$. Since $p_j a$ is left-partitionable,
		\begin{align*}
			\blue{Q \cap \prava{p_ja}} - \red{Q \cap \prava{p_ja}} &\le j-1.
		\end{align*}
		
		Thus,
		\[- \red{Q \cap \prava{p_ja}} \le j-1 -\blue{Q \cap \prava{p_ja}} = n(ap_j).\]
		
		\begin{sloppypar}
			Similarly, since $p_{j+1} a$ is right-partitionable, we have $- \red{Q \cap \leva{p_{j+1}a}} \le n(ap_{j+1})$. Additionally, since $a$ in on the boundary of $Q$, the interior of $\leva{p_{j+1}a} \cap \prava{p_{j}a}$ is entirely outside $Q$. Thus, 
			\begin{align*}
				n(ap_{j+1}) + n(ap_{j}) &\ge - \red{Q \cap \prava{p_ja}} - \red{Q \cap \leva{p_{j+1}a}} \\
				&= -\red{Q \cap  (\prava{p_ja} \cup  \leva{p_{j+1}a})}.
			\end{align*}
		\end{sloppypar}

		The remaining conditions follow immediately from these ones since $n(p_jp_{j+1})$ is a positive number.
		Thus, by Lemma~\ref{chp01:cases3}, $Q$ can be partition into convex polygonal regions $O_1, O_2, O_3$ such that $ap_j, p_jp_{j+1}, ap_{j+1}$ are diagonals of $O_1, O_2, O_3$, respectively, and $\bluedashed{O_1} - \red{O_1}, \bluedashed{O_2} - \red{O_2}, \bluedashed{O_3} - \red{O_3}$ are equal to $n(ap_j), n(p_jp_{j+1}), n(ap_{j+1})$, respectively. See Figure~\ref{fig3:inductionstep} for an illustration. Hence,
		
		\begin{align*}
			\blue{O_1} - \red{O_1} &= n(ap_j) + \blue{Q \cap \leva{ap_j}} = j-1 \\
			\blue{O_2} - \red{O_2} &= n(p_jp_{j+1}) + \blue{Q \cap \leva{p_jp_{j+1}}} = 1 \\
			\blue{O_3} - \red{O_3} &= n(ap_{j+1}) + \blue{Q \cap \leva{p_{j+1}a}} = s-j. \\
		\end{align*}
		
		Therefore, we can apply the induction hypothesis to the polygonal region $O_1$, point~$a$ on the boundary of $O_1$, points $p_1, \dots, p_{j}$, the set of blue points $B \cap O_1$, and the set of red points $R \cap O_1$ (if some red or blue point are on the boundary of $O_1$, we include only the ones assigned to $O_1$ by Lemma~\ref{chp01:cases3}) to obtain partition $(Q_1, \dots, Q_{j-1})$, of $O_1$. Similarly, we obtain a partition $(Q_{j+1}, \dots, Q_{s})$ of $O_3$ by induction hypothesis applied to $O_3$ and corresponding points.
		
		Furthermore, since $n(p_jp_{j+1}) = 1$, $p_j,p_{j+1}$ is a diagonal of $O_2$ and we can set $Q_j = O_2$. Finally, $(Q_1, \dots, Q_s)$ is the desired partition of $Q$.
	\end{proof}
	
	It remains to apply this Lemma to prove Theorem~\ref{chp01:partition}. The idea is to do the first step similarly to the induction step of this lemma. The only real difference is that previously we had at least one vertex of $P$ on the boundary of $Q$. Without it, the Conditions \eqref{cp01:conjectureeq} are not guaranteed to hold and we have to be a bit more careful with the first splitting.
	
	\begin{proof}[Proof of Theorem~\ref{chp01:partition}]
		If $s=3$ we can apply Lemma~\ref{chp01:cases3} and we are done. 
		
		For every $p_i$, $2\le i \le s$, we say that $p_ip_1$ is left-partitionable if \[\blue{\leva{p_ip_1}} - \red{\leva{p_ip_1}} \ge s -i +1\] or, equivalently, if \[\blue{\prava{p_ip_1}} - \red{\prava{p_ip_1}} \le i-1.\] This equivalence holds since $s = |B| - |R|$.

		Similarly, we say that $p_ip_1$ is right-partitionable if \[\blue{\prava{p_ip_1}} - \red{\prava{p_ip_1}} \ge i-1\] or, equivalently, if \[\blue{\leva{p_ip_1}} - \red{\leva{p_ip_1}} \le s -i +1.\] Furthermore, every such $p_ip_1$ is left or right-partitionable. 
		
		%If $p_3p_1$ is right-partitionable or $p_{s-1}p_1$ is left-partitionable then we proceed exactly the same as in the proof of Lemma~\ref{chap01:partiitoninduction} (see Figure~\ref{fig2.5:inductionstep_basic}). Hence we can assume that $p_3p_1$ is left-partitionable, and $p_{s-1}p_1$ is right-partitionable. Therefore, we can find $j$, $3\le j \le s-2$ such that $p_jp_1$ is left-partitionable and $p_{j+1}p_1$ is right-partitionable.
		
		Note that $p_2p_1$ is left-partitionable because it is an edge of $P$, and so there are no blue points to the right of $p_2p_1$. Similarly, $p_sp_1$ is right-partitionable. Therefore, we can find $j$, $2\le j \le s-1$ such that $p_jp_1$ is left-partitionable and $p_{j+1}p_1$ is right-partitionable.
		
		\begin{sloppypar}
			Let $T$ be the triangle $p_1p_jp_{j+1}$. Let $B' = B \cap T$, $n(p_1p_j) = j-1 - \blue{\prava{p_jp_1}}$,  $n(p_{j+1}p_1) = s - j   - \blue{\leva{p_{j+1}p_1}}$, and  $n(p_jp_{j+1}) = 1$. The situation is almost identical as in the proof of Lemma~\ref{chap01:partiitoninduction}. Thus, we would like to use Lemma~\ref{chp01:cases3} on the whole plane, triangle $T$, the set of blue points $B'$, the set of red points $R$, and the numbers $n(p_1p_j), n(p_1p_{j+1}), n(p_jp_{j+1})$. All the assumptions needed in Lemma~\ref{chp01:cases3} are satisfied in the similar way as in the proof of Lemma~\ref{chap01:partiitoninduction} with the exception of the following condition: 
			\begin{align*}
				n(p_1p_{j}) + n(p_1p_{j+1}) &\ge -\red{(\prava{p_jp_1} \cup  \leva{p_{j+1}p_1})}.
			\end{align*}
		\end{sloppypar}

		Hence, if this inequality holds, we can partition the plane into three convex polygons, two of which we can further partition by Lemma~\ref{chap01:partiitoninduction}. In this way, we obtain the desired partition of the plane.
		
		Now assume that 
		\begin{align*}
			n(p_1p_{j}) + n(p_1p_{j+1}) &< -\red{\prava{p_jp_1} \cup  \leva{p_{j+1}p_1}}.
		\end{align*}
		Note that since $p_{j+1}p_1$ is right-partitionable, $p_jp_1$ is left-partitionable, this can happen only if $3\le j \le s-2$ (e.g. in case $j=s-2$, we would have $n(p_{s-1}p_1) = 1$ and the previous inequality would be in contradiction with left-partitionability of $p_{s-2}p_1$).
		
		We can further split the union $\prava{p_jp_1} \cup  \leva{p_{j+1}p_1}$ and write
		\begin{align} \label{chp01:eq1}
			\begin{split}
				n(p_1p_{j}) &+ n(p_1p_{j+1}) < -\red{\prava{p_jp_1} \cap  \prava{p_{j+1}p_1}} -\\&- \red{\leva{p_jp_1} \cap  \leva{p_{j+1}p_1}} 
				-\red{\prava{p_jp_1} \cap  \leva{p_{j+1}p_1}}.
			\end{split}
		\end{align}

		Since $p_jp_1$ is left-partitionable, $\blue{\prava{p_jp_1}} - \red{\prava{p_jp_1}} \le j-1$. By substituting $n(p_1p_j)$ for $j-1 - \blue{\prava{p_jp_1}}$, we have	
		\begin{equation}\label{chp01:eq2}
			n(p_1p_j) \ge -\red{\prava{p_jp_1}}.
		\end{equation}
		
		Similarly, since $p_{j+1}p_1$ is right-partitionable, $n(p_1p_{j+1}) \ge -\red{\leva{p_{j+1},p_1}}$. By combining this with Equation \eqref{chp01:eq1} we get
		\begin{equation}\label{chp01:eq3}
			n(p_1p_j) < -\red{\prava{p_jp_1} \cap  \prava{p_{j+1}p_1}}.
		\end{equation}
		
		Equations \eqref{chp01:eq2} and \eqref{chp01:eq3} imply that there exists a directed half-line $l$ starting at $p_1$ that splits the region $\prava{p_jp_1} \cap  \leva{p_{j+1}p_1} $ in a way that $l$ does not cross any point of $R \cup B$, and
		\begin{equation}\label{chp01:eq4}
			n(p_1p_{j}) = -\red{\prava{p_jp_1} \cap  \prava{l}}. 
		\end{equation}
		
		This together with Equation \eqref{chp01:eq1} implies that
		
		\begin{align}\label{chp01:eq5}
			n(p_1p_{j+1}) < - \red{\leva{l} \cap  \leva{p_{j+1}p_1}}.
		\end{align}
		
		\begin{figure}[htb]
			\centering
			\includegraphics[scale=1]{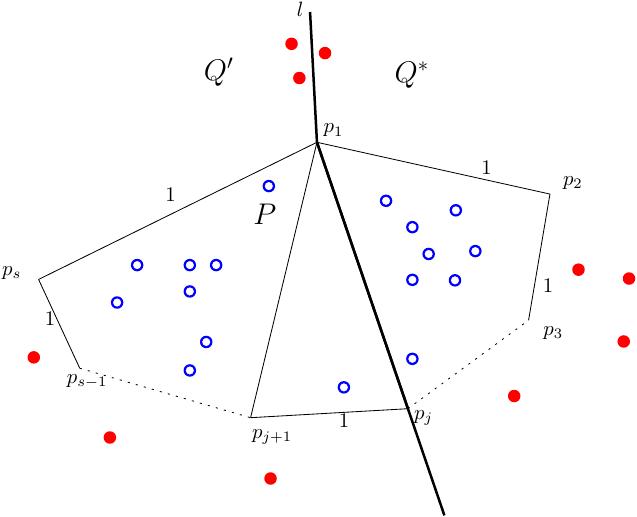}
			\caption{ First step of partitioning the convex polygon $P$.}
			\label{fig4:partitioningproblem}
		\end{figure}
		Therefore, the plane is split into two polygonal regions; convex one $Q^* = \prava{p_jp_1} \cap  \prava{l}$ and non-convex one $Q' = \leva{l} \cup  \leva{p_{j}p_1}$. See Figure~\ref{fig4:partitioningproblem}. If we substitute for the value of $n(p_1p_{j})$ into Equation \eqref{chp01:eq4} we see that  			
		\[j-1 - \blue{\prava{p_jp_1}} = -\red{Q^*}.\] 			
		Furthermore, note that the blue points are only inside $P$. Thus, 
		
		\begin{equation} \label{chp01:eq6}
			j-1 = \blue{Q^*} -\red{Q^*}
		\end{equation}
		
		Hence, we can apply Lemma~\ref{chap01:partiitoninduction} to the polygonal region $Q^*$ with point $p_1$ on its boundary, points $p_1,\dots, p_j$, set of blue points $B \cap Q^*$ and set of red points $R \cap Q^*$. This way we obtain a partition of $Q^*$ into convex polygonal regions $Q_1, \dots, Q_{j-1}$.
		
		Thus, it suffices to partition $Q'$. Equation \eqref{chp01:eq6} together with the fact that $s = |B| - |R|$ implies that $s-j+1 = \blue{Q'} -\red{Q'}$.
		We would like to apply Lemma~\ref{chap01:partiitoninduction} to the polygonal region $Q'$ with point $p_1$ on its boundary, points $p_j,\dots, p_s, p_1$, set of blue points $B \cap Q'$ and set of red points $R \cap Q'$.  The only problem is that $Q'$ is not convex. Luckily, it turns out that it does not matter in this case because if we apply the same approach, the parts of the partition will still be convex. Since the proof is almost the same as in the proof of Lemma~\ref{chap01:partiitoninduction} we present a more concise version.
		
		Let $R' = R \cap Q'$ and $B' = B \cap Q'$. The important part is that Equation \eqref{chp01:eq5} implies that the diagonal $p_{j+1}p_1$ is left-partitionable in $Q'$. Furthermore, $p_s \ne p_{j+1}$ since $s \ge j+2$. Therefore, we once again find some index $k$, $j+1 \le k \le s-1$, such that $p_kp_1$ is left-partition-able and $p_{k+1}p_1$ is right-partitionable. 
		Let $T'$ be the triangle $p_1p_kp_{k+1}$. Let $n(p_1p_k) = k-j - \bluedashed{\prava{p_kp_1}}$,  $n(p_{k+1}p_1) = s - k   - \bluedashed{\leva{p_{k+1}p_1}}$, and  $n(p_kp_{k+1}) = 1$. We apply Lemma~\ref{chp01:cases3} to the whole plane, triangle $T$, set of blue points $B'$, set of red points~$R'$, and the numbers $n(p_1p_k), n(p_1p_{k+1}), n(p_kp_{k+1})$. The Conditions \eqref{cp01:conjectureeq} are satisfied since $p_kp_1$ is left-partition-able, $p_{k+1}p_1$ is right-partitionable, and $\prava{p_kp_1} \cap \leva{p_{k+1}p_1}$ does not contain any point of $B'$ or $R'$ (because its interior is disjoint with $Q'$). Thus, we can partition the plane into convex polygonal regions $O_1, O_2, O_3$ such that $p_1p_k, p_kp_{k+1}, p_1p_{k+1}$ are diagonals of $O_1, O_2, O_3$, respectively, and other conditions about number of blue and red points inside these polygon holds. It is immediate that all intersections $O_i \cap Q'$ are also convex.
		
		Therefore, we can apply Lemma~\ref{chap01:partiitoninduction} to the convex polygonal region $Q'\cap O_1$ with the point $p_1$ on its boundary, points $p_j, \dots, p_{k}$, the set of blue points $B' \cap O_1$, and the set of red points $R' \cap O_1$ (if some red and blue point are on the boundary of $O_1$, we include only the ones assigned to $O_1$ by Lemma~\ref{chp01:cases3}) to obtain a partition $(Q_j, \dots, Q_{k-1})$, of $Q'\cap O_1$. Similarly, we obtain a partition $(Q_{k+1}, \dots, Q_{s+1})$ of $O_3$ by Lemma~\ref{chap01:partiitoninduction} applied to $Q'\cap O_3$ and corresponding points.
		
		Finally, we can set $Q_k = Q'\cap O_2$, and the partition $(Q_1, \dots, Q_s)$ is the desired partition of the plane.

	\end{proof}
	
	\section{Conclusion of the proof}
	
	To finish the proof of Theorem~\ref{mainpolygon}, we use a result proved by \citet{straitline} about point sets with color classes separated by a line. We use a slightly modified version that easily follows from the proof of the original version. For completeness, we include our variant of the proof in Appendix~\ref{AppendixA}. 
	
	\begin{theorem}[\cite{straitline}] \label{chp01:straitline}
		Let $R$ be a set of red points and $B$ be a set of blue points such that $R \cup B$ is in general position. Assume that $|R| -|B| = 1$ and that there are two points $r_1, r_2 \in R$ such that the line $r_1r_2$ separates $R$ from $B$ and that $r_1, r_2$ are vertices of the convex hull $\conv{R\cup B}$. Then there exists an alternating Hamiltonian path on $R \cup B$ with end vertices $r_1, r_2$.
	\end{theorem}

	It remains to put the pieces together and finish the proof of Theorem~\ref{mainpolygon}. 
	
	\begin{proof}[Proof of Theorem~\ref{mainpolygon}]
		
		%Let $s$ be the number of vertices of $P$ and $(p_1, \dots, p_s)$ be the vertices of $P$. Let $R' = R\setminus P$. That is, $R'$ contains exactly the points of $R$ that are not vertices of~$P$. We will assume that $|R| = |B|$ and we will show that there exists a closed alternating Hamiltonian path on $R \cup B$. If $|R|$ and $|B|$ differ by one, we can add one temporary point of the appropriate color, and after we remove it in the end, we will still have an alternating Hamiltonian path. Therefore, $s = |B| - |R'|$.

        We may assume that $|R| = |B|$ and prove that there exists a closed alternating Hamiltonian path, otherwise we could add one point. Let $R' = R\setminus P$. That is, $R'$ contains exactly the points of $R$ that are not vertices of~$P$. Therefore, $s = |B| - |R'|$.

		By Theorem~\ref{chp01:partition} applied on the polygon $P$, the set of blue points $B$ and the set of red points $R'$, there exists a partition of the plane into convex polygonal regions $Q_1, \dots , Q_{s}$ such that for every $i$, the edge $p_ip_{i+1}$ is a diagonal of $Q_i$, and for every $i$, we have $\blue{Q_i} - \reddashed{Q_i} = 1$ (index arithmetic is modulo $s$). Moreover, every point of $R' \cup B$ is counted in exactly one $Q_i$. That is, if a point of $R' \cup B$ is on boundaries of more $Q_i$'s it is assigned to exactly one of them.
		
		\begin{figure}[htb]
			\centering
			\includegraphics[scale=1]{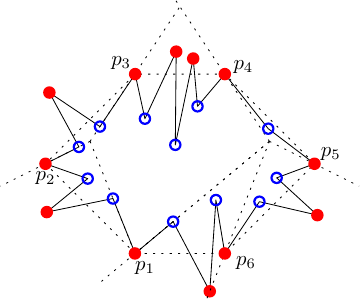}
			\caption[A closed alternating Hamiltonian path.]{A closed alternating Hamiltonian path in a case when $6$ red points form a polygon separating the remaining $6$ red points from $12$ blue points lying inside the polygon.}
			\label{fig6:pathonpolygon}
		\end{figure}
		
		We apply Theorem~\ref{chp01:straitline} to each $Q_i$ separately. Each $Q_i$ contains one more blue point than it contains points of $R'$. Moreover, these red and blue points are separated by line $p_ip_{i+1}$. Additionally, $p_i $ and $ p_{i+1}$ are red points. Since they form a diagonal of the convex $Q_i$, they must be vertices of the convex hull of our red and blue points. Thus, by Theorem~\ref{chp01:straitline} there exists an alternating Hamiltonian path with ends in $p_i$ and $p_{i+1}$ covering all these red and blue points inside $Q_i$. Note that this path is wholly inside the convex polygonal region~$Q_i$ since it consists of straight-line segments connecting points inside a convex polygonal region.

		These paths are connected together in the end vertices $p_i$. Therefore, together they form a closed alternating Hamiltonian path. This cycle is non-crossing since each path is in its own convex polygon $Q_i$, points are in general position, and every point of $R' \cup B$ is assigned to exactly one $Q_i$. See Figure~\ref{fig6:pathonpolygon} for an illustration. 
	\end{proof}

	\bibliographystyle{splncs04nat}
	\bibliography{mybibliography}

	\newpage
	\section*{Appendix}
	\appendix
	
	\section{Proof of Theorem~\ref{chp01:straitline}} \label{AppendixA}
	\begin{proof}
		Denote the separating line $r_1 r_2$ by $s$. By rotating the whole plane, we can assume that $s$ is a vertical line, $R$ is on the left side of $s$, and $B$ is on the right side of $s$. Let $T = R \cup B$. For every subset $X$ of $T$ containing at least one point of each color, there exist one or two alternating edges of $\conv{X}$ intersection $s$. We call the top one \emph{top alternating edge} of $\conv{X}$ and the bottom one \emph{bottom alternating edge} of $\conv{X}$. Since $r_1, r_2$ are vertices of the convex hull of $T$ they are vertices of some alternating edges of $\conv{T}$. Without a loss of generality let $r_1$ be a vertex of the top alternating edge and $r_2$ be a vertex of the bottom alternating edge.
		
		We will inductively build a sequence $(p_1, p_2, \dots, p_{|T|})$ containing all points of $T$ such that connecting these points in the given order by straight-line segments forms an alternating Hamiltonian path. We set $p_1$ to be $r_1$ (it is one of the vertices of the top alternating edge of $\conv{T}$). 
		
		\begin{figure}[htb]
			\centering
			\includegraphics[scale=1]{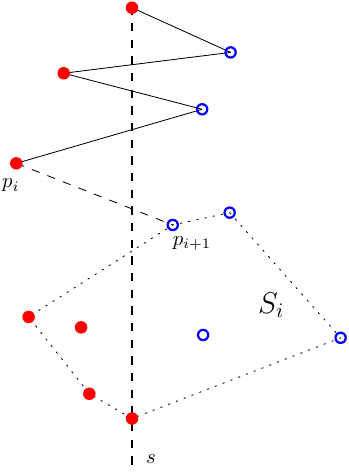}
			\caption[Inductive forming of an alternating Hamiltonian path.]{Inductive forming of an alternating Hamiltonian path when the color classes are separated by a line.}
			\label{fig5:separatedpath}
		\end{figure}
		
		Assume that we have already built the sequence up to $p_i$. Let $S_i$ be the set containing the remaining points. We set $p_{i+1}$ to be the vertex of the top alternating edge of $\conv{S_i}$ colored by a different color than $p_i$. If $|S_i| = 2$ we also set $p_{|T|}$ to be the last unselected point and the path is complete. Clearly, the formed path is alternating. Moreover, in every step the edge $p_ip_{i+1}$ lies in $\conv{S_i \cup \{p_i\}} \setminus \conv{S_i}$ since $p_i$ is a vertex of the top alternating edge of $S_i \cup \{p_i\}$ and $p_{i+1}$ is a vertex of the top alternating edge of $S_i$, see Figure~\ref{fig5:separatedpath} for illustration. Therefore, the formed path is also non-crossing.
		
		It is clear from the construction that the first vertex is a vertex of the top alternating edge of $\conv{T}$. Furthermore, $r_2$ is always a vertex of the bottom alternating edge of $S_i$. Hence, it cannot be part of the top alternating edge of any $S_i$ if there are at least three points left. Since the second to last point has to be from $B$ it means that $r_2$ is the last vertex of our formed path.
	\end{proof}

\end{document}